\documentclass[12pt]{amsart}
\usepackage{amsmath}
\usepackage[active]{srcltx}
\usepackage[pagebackref, colorlinks, linkcolor=red, citecolor=blue, urlcolor=blue, hypertexnames=true]{hyperref}
\usepackage[backrefs]{amsrefs}
\usepackage{enumerate}
\allowdisplaybreaks
\usepackage[all,cmtip]{xy}

\usepackage{color}

\newtheorem{thm}{Theorem}[section]

\newtheorem{lem}[thm]{Lemma}

\newtheorem{definition}[thm]{Definition}
\newtheorem{example}[thm]{Example}
\newtheorem{remark}[thm]{Remark}

\theoremstyle{definition}

\begin{document}
     \title[Shadowable points of free semigroup actions]{Shadowable points of free semigroup actions}
     \author[Ritong Li, Dongkui Ma, Jieyan Tang and Xiaojiang Ye]{}

     \email{summerleslieli@163.com}
     \email{dkma@scut.edu.cn}
     \email{kuangrui@scut.edu.cn}
     \email{yexiaojiang12@163.com}

     \date{\today}
     \thanks{{$^{*}$}Corresponding author: Xiaojiang Ye}

     \maketitle
     \centerline{\scshape Ritong Li, Dongkui Ma, Rui Kuang and Xiaojiang Ye$^*$}
     \medskip
     {\footnotesize
	  \centerline{School of Mathematics, South China University of Technology, }
   	  \centerline{Guangzhou 510641, P.R. China}
   }

    \hspace{2mm}
	\begin{abstract}
    The shadowable points of dynamical systems has been well-studied by Morales \cite{MR3535492}. This paper aims to generalize the main results obtained by Morales to free semigroup actions. To this end, we introduce the notion of shadowable points of a free semigroup action. Let $G$ be a free semigroup generated by finite continuous self-maps acting on compact metric space. We will prove the following results for $G$ on compact metric spaces. The set of shadowable points of $G$ is a Borel set. $G$ has the pseudo-orbit tracing property (POTP) if and only if every point is a shadowable point of $G$. The chain recurrent and non-wandering sets of $G$ coincide when every chain recurrent point is a shadowable point of $G$. The space $X$ is totally disconnected at every shadowable point of $G$ under certain condition. 

	\end{abstract}

	\maketitle

	\section{ Introduction }
	The shadowing theory is an important component of the qualitative theory of dynamical systems, producing many interesting and profound results. If any pseudo orbit with sufficient accuracy approaches a certain precise trajectory, then the dynamical system has the shadowing property or pseudo-orbit tracing property(POTP). So far, many researchers have studied homeomorphisms with shadowing property, for examples \cite{MR1289410, MR1727170}. Recently, Morales \cite{MR3535492} introduced the concept of shadowable points by decomposing shadowing property into points. Subsequently Kawaguchi \cite{MR3724498} extended the notion for continuous maps and gave a quantitative version of a Morales'result in \cite{MR3535492}. In addition, Kawaguchi also obtained more and better results. Bahabadi \cite{MR3353566} introduced the definitions of shadowing and average shadowing properties for iterated function systems, and also proved that an iterated function system has the shadowing property if and only if the skew-product transformation corresponding to the iterated function system has the shadowing property. Hui and Ma \cite{MR3842255}, Zhu and Ma \cite{MR4200965} introduced the definitions of $\delta$-pseudo-orbit for free semigroup actions and the pseudo-orbit tracing property for free semigroup actions.
	
	\quad In this paper, we introduce the notion of shadowable points of free semigroup actions. The main work is to extend Theorem 1.1 and Theorem 1.2 of \cite{MR3535492} to free semigroup actions, and improve the proof process of Lemma 2.1, Lemma 2.3 and Theorem 1.2 in \cite{MR3535492} using Lemma 2.2 in \cite{MR3724498} and new conditions. By studying relationship between free semigroup actions and skew-product transformation corresponding to the free semigroup action, we find out relation between the shadowable points of free semigroup actions and the totally disconnected points. Let $X$ be a compact metric space, $f_0,\cdots,f_{m-1}$ continuous self-maps on $X$. Denote $G$ the free semigroup generated by $G_1=\{f_0,\cdots,f_{m-1}\}$ on $X$ and $F:\Sigma_m^+\times X\to\Sigma_m^+\times X$ a skew-product transformation, where $\Sigma_m^+$ denotes the one-side symbol space. The main results of this paper are:

	\begin{thm}\label{thm1.1}
		Let $(X,d)$ be a compact metric space. If $f_0,f_1,\cdots,f_{m-1}$ are continuous self-maps on X and $G$ is the free semigroup generated by $\{f_0,f_1,\cdots,f_{m-1}\}$, then
		\begin{enumerate}
		\item $G$ has the $POTP$ if and only if $Sh(G)=X$;
		\item $Sh(G)$ is a Borel set;
		\item if $CR(G) \subset Sh(G)$, then $CR(G)=\Omega(G)$.
		\end{enumerate}
    Where $POTP$ is the abbreviation of pseudo-orbit tracing property, $Sh(G)$ denotes the set of shadowable points of $G$, $CR(G)$ denotes the set of chain recurrent points of $G$, and $\Omega(G)$ denotes the set of non-wandering points of $G$.
	\end{thm}
	
	\begin{thm}\label{thm1.2}
        Let $f_0,\cdots,f_{m-1}$ be continuous self-maps of a compact metric space $X$, $G$ the free semigroup generated by $\{f_0,\cdots,f_{m-1}\}$ and $F$ the skew-product transformation corresponding to $\{f_0,\cdots,f_{m-1}\}$. If there exists $\delta>0$ satisfying $B(\gamma(Sh(F)),\delta)\subset R(F)$, then $Sh(G) \subset X^{deg}$.\\
        Where $\gamma(Sh(F))$ denotes the union of all connected components of $X$ whose intersection with $Sh(F)$ is non-empty, $X^{deg}$ denotes the set of all totally disconnected points of $X$, $B(\gamma(Sh(F)),\delta)$ denotes the $\delta$ neighborhood of set $\gamma(Sh(F))$, and $R(F)$ denotes the set of recurrent points of $F$.
	\end{thm}

	This paper is organized as follows. In section 2, we give some preliminaries. In section 3, we prove our main results. These results extend and improve the work of Morales \cite{MR3535492}.

	\section{Preliminaries}
	
	We first introduce some basic notions. Let $(X,d)$ be a compact metric space and $f:X \to X$ be a continuous map. A sequence $\{x_i\}_{i=0}^{\infty}$ is called a $\delta$-pseudo-orbit of $f$ if for each $i\ge0$,
	\[
    d(f(x_i),x_{i+1}) \le \delta.
    \]	

    The continuous map $f$ is said to have the pseudo-orbit tracing property(POTP) if for each $\varepsilon >0$, there exists $\delta>0$ such that every $\delta$-pseudo-orbit $\{x_i\}_{i=0}^{\infty}$ is $\varepsilon$-shadowed by the orbit of some point $y\in X$, i.e. for all $i\ge0$,
    \[
    d(f^i(y),x_i)\le\varepsilon.
    \]
    We say that a point $x\in X$ is $shadowable$ point of $f$ if for every $\varepsilon>0$, there is $\delta(x,\varepsilon)>0$ such that every $\delta$-pseudo-orbit $\{x_i\}_{i=0}^{\infty}$ for $f$ with $x_0=x$ can be $\varepsilon$-shadowed. We denote by $Sh(f)$ the set of shadowable points of $f$. In \cite{MR3724498}, the author defined quantitative shadowable points. For $b>0$, a point $x\in X$ is called $b$-$shadowable$ point of $f$ if there exists $\delta>0$ for which every $\delta$-pseudo-orbit $\{x_i\}_{i=0}^{\infty}$ for $f$ with $x_0=x$ is $b$-shadowed by some point of $X$. Denote by $Sh_b^+(f)$ the set of $b$-shadowable points of $f$. Then, for $c\ge0$, defined\\
    $$Sh_{c+}^+(f)=\bigcap\limits_{b>c}Sh_b^+(f).$$

    Denote by $F_m^{+}$ the set of all finite words of symbols $0,1,\cdots,m-1$. For any $w\in F_m^{+}$, $|w|$ stands for the length of $w$, that is, the number of symbols in $w$. We write $w\le w'$ if there exists a word $w''\in F_m^{+}$ such that $w'=w''w$.

    Denote by $\Sigma_m^+$ the set of all one-side infinite sequences of symbols $0,1,\cdots, m-1$, i.e.
    \[
    \Sigma_m^+=\{\omega=(w_0,w_1,\cdots)~|w_i\in\{0,1,\cdots,m-1\} ~for~all~i\ge 0\}.
    \]

    Let $\omega \in \Sigma_m^+$, $w \in F_m^+$, $a,b\ge0$, and $a\ge b$. We write $\omega|_{[a,b]}=w$ if $w=w_aw_{a+1}\cdots w_{b-1}w_b$. For $\omega=(w_0,w_1,\cdots)\in \Sigma_m^+$, denote $(\omega)_i=w_i$, $i\ge 0$.

    Let $w \in F_m^+, w=w_1w_2\cdots w_k$, where $w_i\in\{0,1,\cdots,m-1\}$ for all $i\in\{1,\cdots,k\}$, and denote $\bar{w}=w_kw_{k-1} \cdots w_1$. Let $f_w=f_{w_1}\circ f_{w_2}\circ\cdots\circ f_{w_k}$. Obviously, $f_{ww'}=f_w\circ f_{w'}$, and $f_{\bar{w}}=f_{w_k}\circ f_{w_{k-1}}\circ\cdots\circ f_{w_1}$.
    	
    \quad In \cite{MR3353566}, the author introduced the shadowing property of the semigroup action generated by two continuous maps on compact metric space $X$. Here we can  generalize it to general case. Let $(X,d)$ be a compact metric space and $f_0,f_1,\cdots,f_{m-1}$ be continuous self maps on X. Denote $G$ the free semigroup acting on the space $X$ generated by $G_1=\{f_0,\cdots,f_{m-1}\}$.
    For $\omega=(w_0,w_1,\cdots)\in \Sigma_m^+$, an orbit of $x\in X$ under $G$ for $\omega$ is a sequence $\{f_{\omega}^n(x)\}_{n=0}^{\infty}$, where
    $$
    \left\{
    \begin{aligned}
    f_{\omega}^n(x)&=f_{w_{n-1}}\circ f_{w_{n-2}}\circ\cdots\circ f_{w_1}\circ f_{w_0}(x).&n\ge 1\\
    f_{\omega}^0(x)&=x.&n=0\\
    \end{aligned}
   \right.
   $$
    \quad The following definition of pseudo-orbit and pseudo-orbit tracing property for $G$ are refer to \cite{MR3353566} and \cite{MR3842255}.

    \begin{definition}[\cite{MR3842255}]
    Given $\delta>0$, a $\delta$-pseudo-orbit($\delta$-chain) $\{x_i\}_{i=0}^{\infty}$ for $G$ is defined as
    \begin{equation}\label{2.1}
    	min\{d(f_0(x_i),x_{i+1}),\cdots,d(f_{m-1}(x_i),x_{i+1})\}<\delta~for~every~i\ge0.
    \end{equation}
    Also, we can express the above expression as follows: there exists $\omega=(w_0,w_1,\cdots)\in \Sigma_m^+$ such that
    \begin{equation}\label{2.2}
	d(f_{w_i}(x_i),x_{i+1})<\delta~for~every~i\ge 0.
    \end{equation}
    \end{definition}

    We say that a $\delta$-pseudo-orbit $\{x_i\}_{i=0}^{\infty}$ for $G$ is $\varepsilon$-shadowed by a point $z\in X$ if for some $\omega$ satisfying (\ref{2.1}) or (\ref{2.2}), it holds that

    \begin{equation} \label{2.3}
    d(f_{\omega}^n(z),x_n)<\varepsilon~for~all~n\ge0.
    \end{equation}
    \begin{definition}[\cite{MR3842255}]
    The free semigroup action $G$ has the pseudo-orbit tracing property(POTP) if for any $\varepsilon>0$, there is $\delta(\varepsilon)>0$ such that every $\delta$-pseudo-orbit for $G$ can be $\varepsilon$-shadowed by some point of $X$.
    \end{definition}

    \begin{remark}
    The strict inequality in (\ref{2.2}) and (\ref{2.3}) can be relaxed to an inequality without affecting the proof of the following lemmas and theorems except for Lemma \ref{lemma3.4}. So the following default definitions are strict inequalities, except for Lemma \ref{lemma3.4}.
    \end{remark}
    \begin{remark}
    If $G$ has the POTP, then maps $f_0,f_1,\cdots,f_{m-1}$ all have POTP by Theorem 1.4 of \cite{MR3353566}, but the converse is not true. Please refer to Example 1.5 of \cite{MR3353566} for specific example.
    \end{remark}

    \begin{example}[\cite{MR3353566} Example 1.2]
    Define two continuous maps $f_0,f_1$ on $\Sigma_2^+$ as follows:
    \[
    f_0(s_0s_1s_2\cdots)=0s_0s_1s_2\cdots, f_1(s_0s_1s_2\cdots)=1s_0s_1s_2\cdots.
    \]
    Denote $G$ the free semigroup generated by $G_1=\{f_0,f_1\}$, then $G$ has the POTP.
    \end{example}
   \quad Next, we introduce two new definitions based on \cite{MR3535492} in this paper.
    \begin{definition}
    A point $x\in X$ is called $shadowable$
   point of $G$ if for every $\varepsilon>0$, there is $\delta(x,\varepsilon)>0$ such that every $\delta$-pseudo-orbit $\{x_i\}_{i=0}^{\infty}$ for $G$ with $x_0=x$ can be $\varepsilon$-shadowed. We denote by $Sh(G)$ the set of shadowable points of $G$.
    \end{definition}
    \begin{remark}
    Clearly, if $G$ has the POTP, then $Sh(G)=X$(i.e.every point is a shadowable point of $G$). The converse is true on compact metric spaces by Theorem \ref{thm1.1}.
    \end{remark}

	Let $X$ be a compact metric space. We say that a sequence $\{x_i\}_{i= 0}^{\infty}$ is $through$ subset $K\subset X$ if $x_0\in K$.

   \begin{definition}
	The free semigroup action $G$ has the POTP $through~K$ if for every $\varepsilon>0$ there is $\delta>0$ such that every $\delta$-pseudo-orbit for $G~through~K$ can be $\varepsilon$-shadowed.
   \end{definition}

For $G$, we assign the following skew-product transformation $F: \Sigma_m^+ \times X \to \Sigma_m^+ \times X$ defined as
\[
F(\omega, x) = (\sigma \omega, f_{\omega_0}(x)),
\]
where $\omega = (w_0, w_1, \cdots)\in\Sigma_m^+$ and $\sigma $ is the shift map. Here $f_{w_0}$ stands for $f_0$ if $w_0=0$, and for $f_1$ if $w_0=1$, and so on. Let $\omega= (w_0, w_1, \cdots)\in \Sigma_m^+$, then
\[
\begin{aligned}
	F^n(\omega,x)&=(\sigma^n\omega,f_{w_{n-1}}\circ f_{w_{n-2}}\circ\cdots\circ f_{w_0}(x))\\
	&=(\sigma ^n\omega,f_{\omega}^n(x)).
\end{aligned}
\]
The metric on $\Sigma_m^+$ is defined as
\[
d_1(\omega,\omega')=\frac{1}{2^k},
\]
where $k=min\{j:w_j\neq w_j'\}$, $\omega=(w_0,w_1,\cdots)$, $\omega'=(w_0',w_1',\cdots)\in \Sigma_m^+$.
A metric on $\Sigma_m^+\times X$ is defined as follows:
\[
D((\omega,x),(\omega',x'))=max\{d_1(\omega,\omega'),d(x,x')\}
\]
for $(\omega,x),(\omega',x')\in \Sigma_m^+\times X$, where $d_1$ and $d$ are metrics on $\Sigma_m^+$ and $X$ respectively.

    \quad The following definitions refer to \cite{tang2022chain} and \cite{MR4200965}.\\
    We say that a point $x\in X$ is a $non-wandering$ point of $G$ if for every neighbourhood $U$ of $x$ there is $\omega\in \Sigma_m^+$ and $k\in \mathbb{N}$ such that $(f_{\omega}^k)^{-1}(U)\cap U\neq \emptyset$.

    For $w=w_0w_1\cdots w_{n-1}\in F_m^+$, a $(w,\varepsilon)$-chain (or $(w,\varepsilon)$-pseudo-orbit) of $G$ from $x$ to $y$ is a sequence $\{x_0=x,x_1,\cdots,x_n=y\}$ such that $d(f_{w_i}(x_i),x_{i+1})<\varepsilon$ for $i\in\{0,\cdots,n-1\}$. We say that $x$ is a $chain~recurrent$  point of $G$ if for every $\varepsilon>0$, there is a $(w,\varepsilon)$-chain from $x$ to itself for some $w\in F_m^+$.

    A point $x\in X$ is called a $recurrent$ point of $G$ if there exist $\omega \in \Sigma_m^+$ and an increasing sequence $\{n_i\}_{i\ge0}$ of positive integers such that $\lim\limits_{i\to +\infty}f_{\omega}^{n_i}(x)=x$.
	
	Denote by $\Omega(G)$, $CR(G)$ and $R(G)$ the set of non-wandering, chain recurrent and recurrent points of $G$, respectively. Clearly, we have $R(G)\subset \Omega(G)$, $\Omega(G)\subset CR(G)$.

	On the other hand, the space $X$ is $totally~disconnected~at~p\in X
$ if the connected component of $X$ containing $p$ is $\{p\}$. As in \cite{MR2505316}, we denote
    \[
    X^{deg}=\{p\in X:X~is~totally~disconnected~at~p\}.
    \]
    Recall that $X$ is $totally~disconnected$ if it is totally disconnected at any point(i.e.$X=X^{deg}$). We denote $\gamma(A)$ is the union of all connected components of $X$ whose intersection $A$ is non-empty.

    \quad The following lemmas and theorems are all from references that need to be used in the proof process of this paper.

    \begin{lem}[\cite{MR3724498} Lemma 2.2]\label{lemma2.9}
    Let $f: X\to X$ be a continuous map. If $x\in Sh_{c+}^+(f)$ with $c\ge 0$, then for every $b>c$, there exists $\delta=\delta(x,b)>0$ such that every $\delta$-pseudo orbit $\{x_i\}_{i=0}^{\infty}$ with $d(x,x_0)<\delta$ is $b$-shadowed by some point of $X$.
    \end{lem}
    \begin{lem}[\cite{MR3535492} Lemma 2.1]\label{lemma2.10}
    A homeomorphism of a compact metric space has the POTP through a subset $K$ if and only if for every $\varepsilon>0$ there is $\delta>0$ such that every $\delta$-pseudo-orbit of $f$ through $K$ can be $\varepsilon$-shadowed.	
    \end{lem}
    \begin{lem}[\cite{MR3535492} Lemma 2.2]\label{lemma2.11}
    Let $f$ be a homeomorphism of a compact metric space. Then, for every $z\in\Omega(f)\cap Sh(f)$ and every $\varepsilon>0$ there are $k\in\mathbb{N}^+$ and $y\in X$ such that $f^{pk}(y)\in B[z,\varepsilon]$ for every $p\in\mathbb{Z}$.
    \end{lem}
    \begin{lem}[\cite{MR3535492} Lemma 2.6]\label{lemma2.12}
    If $f: X\to X$ is a homeomorphism of a compact metric space $X$, then $Sh(f)=Sh(f^k)$ for every $k\in \mathbb{Z}\textbackslash\{0\}$.
    \end{lem}

    \begin{thm}[\cite{MR3535492} Theorem 1.2]\label{thm2.13}
    If $f: X\to X$ is a pointwise-recurrent homeomorphism of a compact metric space $X$, then $Sh(f)\subset X^{deg}$.
    \end{thm}
    \begin{lem}[\cite{MR1176513} Chapter IV Lemma 25]\label{lemma2.14}
    If $f: X\to X$ is a homeomorphism of a compact metric space $X$, then for any positive integer $m$, we have $R(f)=R(f^m)$.
    \end{lem}

	\section{the Proof of the Main Result}
    In this section, we give the proof of Theorem \ref{thm1.1} and Theorem \ref{thm1.2}. Before our proof, we must give the following lemmas.
    \begin{lem}\label{lemma3.1}
    If $x\in Sh(G)$, then for every $\varepsilon>0$, there exists $\delta(x,\varepsilon)>0$, such that every $\delta$-pseudo-orbit $\{x_i\}_{i=0}^{\infty}$ for $G$ with $d(x_0,x)<\delta$ is $\varepsilon$-shadowed by some point of $X$.
    \end{lem}

    \begin{proof}
    Since $x\in Sh(G)$, then take $\varepsilon>0$. For $\frac{\varepsilon}{2}>0$, there exists $0<\delta_1(x,\varepsilon)<\varepsilon$, such that every $\delta_1$-pseudo-orbit $\{z_i\}_{i=0}^{\infty}$ for $G$ with $z_0=x$ can be $\frac{\varepsilon}{2}$-shadowed. Take $\delta<\frac{\delta_1}{2}$ such that $d(a,b)<\delta$ implies $d(f_j(a),f_j(b))<\frac{\delta_1}{2}$ for every $a,b\in X$, $j\in \{0,1,\cdots,m-1\}$. Given $\delta$-pseudo-orbit $\{x_i\}_{i=0}^{\infty}$ with $d(x_0,x)<\delta$. Define $\{y_i\}_{i=0}^{\infty}$ by $y_0=x$ and $y_i=x_i$ for $i\ge 1$. Assume $d(f_{w_i}(x_i),x_{i+1})<\delta$, $\omega=(w_0,w_1,\cdots)\in \Sigma_m^+$, then

    $$
    d(f_{w_i}(y_i),y_{i+1})=
    \begin{cases}
    d(f_{w_i}(x_i),x_{i+1})&i\ge 1\\
    d(f_{w_0}(x),x_1)&i=0
    \end{cases}
    $$
    and
    \[
    d(f_{w_i}(x_i),x_{i+1})<\delta<\delta_1,
    \]
    \[
    d(f_{w_0}(x),x_1)\le d(f_{w_0}(x),f_{w_0}(x_0))+d(f_{w_0}(x_0),x_1)<\frac{\delta_1}{2}+\delta<\delta_1.
    \]
    so $\{y_i\}_{i=0}^{\infty}$ is a $\delta_1$-pseudo-orbit for $G$ with $y_0=x$, then there exists $y\in X$, such that $d(f_{\omega}^n(y),y_n)<\frac{\varepsilon}{2}$ for every $n\ge0$. Then, we have
    \[
    d(y,x_0)\le d(y,x)+d(x,x_0)<\frac{\varepsilon}{2}+\delta<\varepsilon,
    \]
    and
    \[
    d(f_{\omega}^n(y),x_n)=d(f_{\omega}^n(y),y_n)<\frac{\varepsilon}{2}<\varepsilon
    \]
    for all $n\ge 1$. Hence, $y$ is a $\varepsilon$-shadowing point of $\{x_i\}_{i=0}^{\infty}$.
    \end{proof}

    \begin{lem}\label{lemma3.2}
    If $f_0,\cdots,f_{m-1}$are continuous self-maps of a compact metric space and $G$ is the free semigroup generated by $\{f_0,\cdots,f_{m-1}\}$, then $G$ has the POTP through a compact subset K if and only if every point in K is a shadowable point of G.
    \end{lem}

    \begin{proof}
    By the previous remark we only have to prove the sufficiency. Now we use the finite cover theorem to prove the conclusion.
    Take $\varepsilon>0$, $x\in K$. There exists $\delta_x>0$ by Lemma \ref{lemma3.1}. Consider the open cover $\{B(x,\delta_x)~|x\in K\}$, obviously it is the open cover of $K$. Since $K$ is compact, there exists a finite open cover  $\{B(x_1,\delta_1),\cdots,B(x_n,\delta_n)\}$ such that $K\subset \bigcup\limits_{j=1}^nB(x_j,\delta_j)$. Let $\delta=\underset{1\leq j \leq n}{min}\{\delta_j\}$. Take $\delta$-pseudo-orbit $\{x_i\}_{i=0}^{\infty}$ for $G$ with $x_0\in K$, then there exists $1\le k\le n$ such that $x_0\in B(x_k,\delta_k)$. Obviously $\{x_i\}_{i=0}^{\infty}$ is also a $\delta_k$-pseudo-orbit. Then $\{x_i\}_{i=0}^{\infty}$ can be $\varepsilon$-shadowed by Lemma \ref{lemma3.1}. So $G$ has the $POTP~through~K$.
    \end{proof}

    \begin{lem}\label{lemma3.3}
    If $f_0,\cdots,f_{m-1}$are continuous self-maps of a compact metric space and $G$ is the free semigroup generated by $\{f_0,\cdots,f_{m-1}\}$, then $CR(G)\bigcap Sh(G)\subset\Omega(G)$.
    \end{lem}

    \begin{proof}
    Take $x\in CR(G)\bigcap Sh(G)$. Since $x\in Sh(G)$, for every $\varepsilon>0$, there exists $\delta>0$ such that every $\delta$-pseudo-orbit $\{z_i\}_{i=0}^{\infty}$ for $G$ with $z_0=x$ can be $\varepsilon$-shadowed. Since $x\in CR(G)$ and $\delta>0$, there is a $(w,\delta)$-chain from $x$ to itself for some $w\in F_m^+$, assume the $(w,\delta)$-chain is $\{x_0=x,x_1,\cdots,x_n=x\}$ with $d(f_{w_i}(x_i),x_{i+1})<\delta$ for $0\le i\le n-1$ and $w=(w_0,w_1,\cdots,w_{n-1})$. Define the sequence $\xi_{kn+r}=x_r$, where $0\le r\le n-1$ and $k\in \mathbb{N}$. Let $v=(w_0,w_1,\cdots,w_{n-1})^{\infty}$. Then we obtain $d(f_{v_i}(\xi_i),\xi_{i+1})<\delta$. So $\{\xi_i\}_{i=0}^{\infty}$ is a $\delta$-pseudo-orbit for $G$ with $\xi_0=x_0=x$, then there exists $y\in X$ such that $d(f_v^j(y),\xi_j)<\varepsilon$ for every $j\in\mathbb{N}$. Then we have
    \[
    d(y,\xi_0)=d(y,x)<\varepsilon,
    \]
    and
    \[
    d(f_v^n(y),\xi_n)=d(f_v^n(y),x)<\varepsilon
    \]
    that is $(f_v^n)^{-1}B(x,\varepsilon)\bigcap B(x,\varepsilon)\neq\emptyset$, so $x\in \Omega(G)$. We obtain $CR(G)\bigcap Sh(G)\subset\Omega(G)$ by the arbitrariness of $x$.
    \end{proof}
    \begin{remark}
    From the proof process of Lemma \ref{lemma3.3}, we can know that the point $y$ satisfies $d(f_v^{kn}(y),x)<\varepsilon$ for every $k\in \mathbb{N}$, so the orbit of point $y$ will periodically enter the neighbourhood of $x$.
    \end{remark}

    \begin{lem}\label{lemma3.4}
    Let $f_0,\cdots,f_{m-1}$ be continuous self-maps of a compact metric space and $G$ the free semigroup generated by $\{f_0,\cdots,f_{m-1}\}$, then Sh(G) is a Borel set.
    \end{lem}

    \begin{proof}
    Given $\varepsilon>0$, let $Sh_{\varepsilon}(G)$ be the set of points $x\in X$ satisfying for $\varepsilon>0$, there exists $\delta_{\varepsilon}>0$ such that every $\delta_{\varepsilon}$-pseudo-orbit $\{z_i\}_{i=0}^{\infty}$ with $z_0=x$ can be $\varepsilon$-shadowed. Also, given $\varepsilon>0$ and $\delta>0$, let $S_{\delta,\varepsilon}(G)$ be the set of points $x\in X$ satisfying for every pseudo-orbit $\{y_i\}_{i=0}^{\infty}$ with $y_0=x$ and $d(f_{w_i}(y_i),y_{i+1})<\delta$ for all $i\ge0$ where $\omega=(w_0,w_1,\cdots)\in\Sigma_m^+$, there exists $y\in X$ such that $d(f_{\omega}^i(y),y_i)\le\varepsilon$ for all $i\ge0$.\\
    Clearly from the definition of $Sh(G)$ we easily have
    \begin{equation}\label{*}
    Sh(G)=\bigcap_{k\in\mathbb{N^+}}Sh_{\frac{1}{k}}(G),
    \end{equation}
    and for every $k\in\mathbb{N^+}$
    \begin{equation}\label{**}
    Sh_{\frac{1}{k}}(G)=\bigcup_{m\in\mathbb{N^+}}S_{\frac{1}{m},\frac{1}{k}}(G).
    \end{equation}
    For any given $k,m\in\mathbb{N^+}$, we prove that $S_{\frac{1}{m},\frac{1}{k}}(G)$ is a closed subset of $X$ for all $m\in\mathbb{N^+}$. Assume $\{z_n\}_{n\in\mathbb{N}}\subset S_{\frac{1}{m},\frac{1}{k}}(G)$ and $\lim\limits_{n\to\infty}z_n=z$ for some $z\in X$. Next we show $z\in S_{\frac{1}{m},\frac{1}{k}}(G)$. Given $\frac{1}{m}$-pseudo-orbit $\{x_i\}_{i=0}^{\infty}$ for $G$ with $x_0=z$ and $d(f_{w_i}(x_i),x_{i+1})<\frac{1}{m}$ where $\omega=(w_0,w_1,\cdots)\in\Sigma_m^+$. For every $n\in\mathbb{N}$, consider sequence $\{x_i^{(n)}\}_{i=0}^{\infty}$ satisfying $x_0^{(n)}=z_n$ and $x_i^{(n)}=x_i$ for $i>0$. Since $d(f_{w_0}(x_0),x_1)<\frac{1}{m}$, there exists $t>0$ such that $d(f_{w_0}(x_0),x_1)<t<\frac{1}{m}$. We can choose $\delta>0$ such that $d(a,b)<\delta$ implies $d(f_j(a),f_j(b))<\frac{1}{m}-t$ for every $j\in\{0,1,\cdots,m-1\}$ and every $a$, $b\in X$. Since $\lim\limits_{n\to\infty}z_n=z$, we have $d(z_n,z)<\delta$ for sufficiently large $n\in\mathbb{N}$, so $d(f_{w_0}(z_n),f_{w_0}(z))<\frac{1}{m}-t$. Then we have
    $$
    \begin{aligned}
    d(f_{w_0}(x_0^{(n)}),x_1^{(n)})=d(f_{w_0}(z_n),x_1)&\le d(f_{w_0}(z_n),f_{w_0}(z))+d(f_{w_0}(z),x_1)\\
    &<\frac{1}{m}-t+t=\frac{1}{m},
    \end{aligned}
    $$
    and for $i>0$
    $$
    \begin{aligned}
    d(f_{w_i}(x_i^{(n)}),x_{i+1}^{(n)})=d(f_{w_i}(x_i),x_{i+1})<\frac{1}{m},
    \end{aligned}
    $$
    when $n$ is large enough. For such $n\in\mathbb{N}$, the pseudo-orbit $\{x_i^{(n)}\}_{i=0}^{\infty}$ is a $\frac{1}{m}$-pseudo-orbit for $G$ with $x_0^{(n)}=z_n\in S_{\frac{1}{m},\frac{1}{k}}(G)$ and $d(f_{w_i}(x_i^{(n)}),x_{i+1}^{(n)})<\frac{1}{m}$, so there exists $y_n\in X$ such that $d(f_{\omega}^i(y_n),x_i^{(n)})<\frac{1}{k}$ for every $i\ge0$. Take s subsequence $\{y_{n_j}\}_{j\in\mathbb{N}}$ and assume $\lim\limits_{j\to\infty}y_{n_j}=y$ for some $y\in X$. Then we have
    \[
    d(y,x_0)=d(y,z)=\lim\limits_{j\to\infty}d(z_{n_j},y_{n_j})=\lim\limits_{j\to\infty}d(x_0^{(n_j)},y_{n_j})\le\frac{1}{k},
    \]
    and for every $i>0$
    \[
    d(f_{\omega}^i(y),x_i)=\lim\limits_{j\to\infty}d(f_{\omega}^i(y_{n_j}),x_i^{(n_j)})\le\frac{1}{k}.
    \]
    Hence $\{x_i\}_{i=0}^{\infty}$ can be $\frac{1}{k}$-shadowed, that is $z\in S_{\frac{1}{m},\frac{1}{k}}(G)$. From this we know that $S_{\frac{1}{m},\frac{1}{k}}(G)$ is a closed subset of $X$, and therefore we obtain $Sh_{\frac{1}{k}}(G)$ is a Borel set in $X$ for every $k\in\mathbb{N^+}$ by (\ref{**}). Thus $Sh(G)$ is a Borel set in $X$ by (\ref{*}).
    \end{proof}

    \begin{lem}\label{lemma3.5}
    If $f_0,\cdots,f_{m-1}$are continuous self-maps of a compact metric space $X$ and $G$ is the free semigroup generated by $\{f_0,\cdots,f_{m-1}\}$, $F$ is the skew-product transformation corresponding to $\{f_0,\cdots,f_{m-1}\}$ which satisfies $Sh(F)=(\Sigma_m^+\times X)^{deg}$, then $Sh(G)=X^{deg}$.
    \end{lem}

    \begin{proof}
    Firstly we prove $Sh(G)\subset X^{deg}$. Take $x\in Sh(G)$ and $\varepsilon>0$. There exists $k\in \mathbb{N}$ such that $\frac{1}{2^k}<\varepsilon$ and $\delta'>0$ such that every $\delta'$-pseudo-orbit $\{z_i\}_{i=0}^{\infty}$ for $G$ with $z_0=x$ can be $\varepsilon$-shadowed. Let $\delta=min\{\frac{1}{2^{k+1}},\delta'\}$, $\omega\in \Sigma_m^+$. Now take a $\delta$-pseudo-orbit $\{\xi_i\}_{i=0}^{\infty}$ of $F$ with $\xi_{i}=(\omega^{(i)},x_i)$ and $\xi_0=(\omega,x)$, where $\omega^{(i)}=(w_0^{(i)},w_1^{(i)},\cdots)$. Therefore
    \[
    D(F(\xi_{i}),\xi_{i+1})=max\{d_1(\sigma\omega^{(i)},\omega^{(i+1)}),d(f_{w_0^{(i)}}(x_i),x_{i+1})\}<\delta~for~all~i\ge 0.
    \]
    So we have
    \[
    d_1(\sigma\omega^{(i)},\omega^{(i+1)})<\delta\le\frac{1}{2^{k+1}}.
    \]
    That is $w^{(i)}_{j+1}=w^{(i+1)}_j$ for every $0\le j\le k+1$, $i\ge0$. Consider $v=(w_0^{(0)},w_0^{(1)},w_0^{(2)},\cdots)$. Since $d(f_{w_0^{(i)}}(x_i),x_{i+1})<\delta<\delta'$ by $D(F(\xi_{i}),\xi_{i+1})<\delta$, and $x_0=x\in Sh(G)$, there exists $y\in X$, such that $d(f_v^n(y),x_n)<\varepsilon$ for all $n\ge 0$. In addition for this $v$, one can see that $d_1(\sigma^n(v),\omega^{(n)})\le\frac{1}{2^{k+3}}<\varepsilon$ for all $n\ge0$. Therefore we have
    \[
    D(F^n(v,y),\xi_{n})<\varepsilon \quad for~all~n\ge0.
    \]
    Hence $(\omega,x)\in Sh(F)=(\Sigma_m^+\times X)^{deg}$. Next, we will prove that $x\in X^{deg}$.

    Suppose by contradiction that $x\notin X^{deg}$. Then the connected component $E$ of $X$ containing $x$(which is compact) has positive diameter diam($E$)>0. Obviously $\{\omega\}\times F$ is also a connected component of $\Sigma_m^+\times X$ containing $(\omega,x)$. This contradicts with $(\omega,x)\in(\Sigma_m^+\times X)^{deg}$. Hence $Sh(G)\subset X^{deg}$.

    \quad On the other hand, for any $x\in X^{deg}$, we claim that for any $\omega\in\Sigma_m^+$ we have $(\omega,x)\in(\Sigma_m^+\times X)^{deg}$. If not,  there exists $v\in\Sigma_m^+$, such that $(v,x)\notin(\Sigma_m^+\times X)^{deg}$. Then we obtain a connected component $E$ of $\Sigma_m^+\times X$ containing $(v,x)$ and $E\textbackslash(v,x)\neq\emptyset$. We divide into two cases:\\
    $\bf Case~1$: If $E\mid_X=\{x\}$, we have $E=A\times\{x\}$ where $A\subset\Sigma_m^+$.  Since $\Sigma_m^+$ is totally disconnected, there exists open sets $U_1,U_2\subset\Sigma_m^+$ such that $(U_1\cap A)\cap(U_2\cap A)=\emptyset$ and $A\subset U_1\cup U_2$, then we have $[(U_1\times\{x\})\cap E]\cap[(U_2\times\{x\})\cap E]=\emptyset$ and $E\subset(U_1\times\{x\})\cup(U_2\times\{x\})$. According to the definition of connectivity we know that $E$ is not connected. It's wrong.\\
    $\bf Case~2$: If $E\mid_X\textbackslash\{x\}\neq\emptyset$, $E\mid_X$ is not connected by $x\in X^{deg}$. There exists open sets $V_1,V_2\subset X$ such that $(V_1\cap E\mid_X)\cap(V_2\cap E\mid_X)=\emptyset$ and $E\mid_X\subset V_1\cup V_2$, then we have $[(\Sigma_m^+\times V_1)\cap E]\cap[(\Sigma_m^+\times V_2)\cap E]=\emptyset$ and $E\subset\Sigma_m^+\times V_1\cup \Sigma_m^+\times V_2$. So $E$ is not connected. It is conflict with $E$ is connected.

    \quad In all, we obtain $E=\{(v,x)\}$, therefore $(v,x)\in (\Sigma_m^+\times X)^{deg}$. It is conflicts with $(v,x)\notin (\Sigma_m^+\times X)^{deg}$, so we have $(\omega,x)\in (\Sigma_m^+\times X)^{deg}=Sh(F)$ for every  $\omega\in\Sigma_m^+$.

    \quad Next, we show $x\in Sh(G)$. Take $\omega\in \Sigma_m^+$ and $0<\varepsilon<\frac{1}{2}$. we have $(\omega,x)\in Sh(F)$ according to the above proof. For $0<\varepsilon<\frac{1}{2}$, there exists $\delta_{\omega}>0$ by Lemma \ref{lemma2.9}, such that
    \[
    \Sigma_m^+\subset \bigcup_{\omega\in\Sigma_m^+}B(\omega,\delta_{\omega}).
    \]
      Actually $\Sigma_m^+$ is compact with respect to metric $d_1$ by Tychonoff theorem. So there exists finite open cover such that
      \[
      \Sigma_m^+\subset\bigcup_{i=1}^nB(\omega^{(i)},\delta_i),
      \]
      where $\omega^{(i)}\in\Sigma_m^+$ and $(\omega^{(i)},x)\in Sh(F)$. Let $\delta=\underset{1\leq i \leq n}{min}\{\delta_i\}$, given $\delta$-pseudo-orbit $\{x_i\}_{i=0}^{\infty}$ for $G$ with $x_0=x$ and $d(f_{v_i}(x_i),x_{i+1})<\delta$ for all $i\ge0$ where $v=(v_0,v_1,\cdots)\in\Sigma_m^+$. Consider sequence $\{\xi_i\}_{i=0}^{\infty}$ with $\xi_i=(\sigma^iv,x_i)$, $\xi_0=(v,x)$. Obiously $\{\xi_i\}_{i=0}^{\infty}$ is a $\delta$-pseudo-orbit of $F$. Since $v\in\Sigma_m^+$, there exists $1\le k\le n$, such that $v\in B(\omega^{(k)},\delta_k)$. Also we have
        \[
        D(\xi_0,(\omega^{(k)},x))=max\{d(x,x),d_1(v,\omega^{(k)})\}<\delta_k.
        \]
        For above $\varepsilon$ and Lemma \ref{lemma2.9}, we know $\{\xi_i\}_{i= 0}^{\infty}$ can be $\varepsilon$-shadowed. There exists $(u,y)\in\Sigma_m^+\times X$ such that
        \[
        D(F^i(u,y),\xi_{i})<\varepsilon,
        \]
        that is $d_1(\sigma^iu,\sigma^iv)<\varepsilon $ and $d(f_u^i(y),x_i)<\varepsilon$ for every $i\ge 0$. Since $d_1(\sigma^iu,\sigma^iv)<\varepsilon<\frac{1}{2}$ we obtain $u=v$. Hence  $d(f_u^i(y),x_i)=d(f_v^i(y),x_i)<\varepsilon$ for all $i\ge 0$, that is $x\in Sh(G)$. Therefore $X^{deg}\subset Sh(G)$.
    \end{proof}

    \quad The following lemma is an improvement of Theorem \ref{thm2.13} and with some modifications on the proof process of Theorem \ref{thm2.13}.
    \begin{lem}\label{lemma3.6}
    Let $f:X\to X$ be a continuous map of a compact metric space $X$. If there exists $\delta>0$ satisfying $B(\gamma(Sh(f)),\delta)\subset R(f)$, then $Sh(f) \subset X^{deg}$.
    \end{lem}

    \begin{proof}
    Take $z\in Sh(f)$ and suppose by contradiction that $z\notin X^{deg}$. Then the connected component $\gamma(\{z\})$, denoted by $E$, is compact and $diam(E)>0$. Take $0<\varepsilon<\{\frac{1}{11}diam(E),\delta\}$. Obviously $z\in B(\gamma(Sh(f)),\delta)$, so $z\in R(f)\subset\Omega(f)$. There exists $k\in\mathbb{N}^+$ and $y\in X$ such that $f^{nk}(y)\in B(z,\varepsilon)$ for every $n\ge0$ by Lemma \ref{lemma2.11}. Define $g=f^k$. Then
    \begin{equation}\label{1}
    g^n(y)\in B(z,\varepsilon)~for~all~n\ge0.
    \end{equation}

    On the other hand $z\in Sh(g)$ by Lemma \ref{lemma2.12}. Then for above $\varepsilon$ ,there exists $\delta'>0$ by Lemma \ref{lemma3.1}. We can assume $\delta'<\varepsilon$. Since $E$ is compact and connected, we can choose a sequence $y=p_1,p_2,\cdots,p_N\in E$ such that $d(p_i,p_{i+1})\le\frac{\delta'}{2}$ for $1\le i\le N-1$ and
    \begin{equation}\label{3}
    E\subset\bigcup\limits_{i=1}^NB(p_i,\delta').
    \end{equation}

    Also $R(g)=R(f^k)=R(f)$ by Lemma \ref{lemma2.14}, and $p_i\in E\subset B(\gamma(Sh(f)),\delta)\subset R(f)=R(g)$. From this we can find positive integers $c(i)~(1\le i\le N)$ such that $d(p_i,g^{c(i)}(p_i))\le\frac{\delta'}{2}$ for all $1\le i\le N$. Since $z\in E\subset\bigcup\limits_{i=1}^NB(p_i,\delta')$, there exists $p_{i_z}~(1\le i_z\le N)$ such that $z\in B(p_{i_z},\delta')$. We define the sequence $\{\xi_i\}_{i=0}^{\infty}$ as follow
    $$
    \begin{aligned}
	\xi_i & =g^i(p_{i_z}) & & \text { if } 0\le i\le c(i_z)-1 , \\
	\xi_{c(i_z)+i} & =g^i(p_{i_z+1}) & & \text { if } 0 \leq i \leq c(i_z+1)-1, \\
	& \vdots & & \\
	\xi_{c(i_z)+\cdots+c(N-2)+i} & =g^i(p_{N-1}) & & \text { if } 0 \leq i \leq c(N-1)-1, \\
	\xi_{c(i_z)+\cdots+c(N-1)+i} & =g^i\left(p_N\right) & & \text { if } 0 \leq i \leq c(N)-1, \\
	\xi_{c(i_z)+\cdots+c(N)+i} & =g^i\left(p_{N-1}\right) & & \text { if } 0 \leq i \leq c(N-1)-1, \\
	& \vdots & & \\
	\xi_{c(1)+\cdots+c(i_z-1)+2\{c(i_z)+\cdots+c(N-1)\}+c(N)+i} & =g^i\left(p_1\right) & & \text { if } i \geq 0 .
    \end{aligned}
    $$
    Obviously $\{\xi_i\}_{i=0}^{\infty}$ is a $\delta'$-pseudo-orbit of $g$ with $\xi_0=p_{i_z}\in B(z,\delta')$. Then there exists $x\in X$ such that $d(g^n(x),\xi_n)<\varepsilon$ for every $n\ge 0$ by Lemma \ref{lemma2.10}. From the definition of $\{\xi_i\}_{i= 0}^{\infty}$ we conclude that there exists integers $n_1,\cdots,n_N$ satisfying $d(g^{n_i}(x),p_i)<\varepsilon$ for every $1\le i\le N$. Take $c=c(1)+\cdots+c(i_z-1)+2\{c(i_z)+\cdots+c(N-1)\}+c(N)$, we obtain
    \[
    d(g^{i+c}(x),g^i(y))<\varepsilon \quad (for~i\ge 0).
    \]
    This combined with (\ref{1}) yields
    \[
    g^i(x)\in B(z,2\varepsilon)~whenever~i\ge c.
    \]
    Since $d(g^{n_i}(x),p_i)<\varepsilon\le\delta$ and $p_i\in E$ for every $1\le i\le N$, we have $g^{n_i}(x)\in B(\gamma(sh(f)),\delta)\subset R(g)$. Thus for $g^{n_i}(x)$, there exists $k_i\ge c$ such that $d(g^{k_i}(x),g^{n_i}(x))<\varepsilon$.
    Now take $w\in E$. It follows from (\ref{3}) that $d(w,p_i)<\delta'$ for some $1\le i\le N$. Then $d(g^{n_i}(x),w)\le d(g^{n_i}(x),p_i)+d(p_i,w)<\varepsilon+\delta'<2\varepsilon$. Now we have two cases:\\
    $\bf Case~1$: If $0\le n_i<c$, then
    $$
    \begin{aligned}
    d(w,z)&\le d(w,g^{n_i}(x))+d(g^{n_i}(x),g^{k_i}(x))+d(g^{k_i}(x),g^{k_i-c}(p_1))+d(g^{k_i-c}(p_1),z)\\
    &<2\varepsilon+\varepsilon+\varepsilon+\varepsilon\\
    &=5\varepsilon.
    \end{aligned}
    $$
    $\bf Case~2$: If $ n_i\ge c$, then
    $$
    \begin{aligned}
    	d(w,z)&\le d(w,g^{n_i}(x))+d(g^{n_i}(x),g^{n_i-c}(p_1))+d(g^{n_i-c}(p_1),z)\\
    	&<2\varepsilon+\varepsilon+\varepsilon\\
    	&=4\varepsilon.
    \end{aligned}
    $$
    We conclude that $E\subset B(z,5\varepsilon)$ by the arbitrariness of $w$ and adove cases, so $diam(E)<10\varepsilon$. This contradicts the choice of $\varepsilon$ therefore $z\in X^{deg}$. As $z\in Sh(f)$ is arbitrary, we obtain $Sh(f)\subset X^{deg}$.
   \end{proof}
   \noindent$\textbf{Proof~of~Theorem~\ref{thm1.1}}$. Let $f_0,f_1,\cdots,f_{m-1}$ be continuous self-maps on compact metric space $X$ and $G$ the free semigroup generated by $\{f_0,f_1,\cdots,f_{m-1}\}$. We have that $Sh(G)$ is a Borel set by Lemma \ref{lemma3.4}. Take $K=X$ in Lemma \ref{lemma3.2}, then we obtain that $G$ has the POTP if and only if $Sh(G)=X$. Finally, since $\Omega(G)\subset CR(G)$ we have that if $CR(G)\subset Sh(G)$, then $CR(G)=\Omega(G)$ by Lemma \ref{lemma3.3}. $\hfill\square$

   \noindent$\textbf{Proof~of~Theorem~\ref{thm1.2}}$. Let $f_0,f_1,\cdots,f_{m-1}$ be continuous self-maps on compact metric space $X$, $G$ the free semigroup generated by $\{f_0,f_1,\cdots,f_{m-1}\}$ and $F$ the skew-product transformation corresponding to $\{f_0,\cdots,f_{m-1}\}$. If there exists $\delta>0$ such that $B(\gamma(Sh(F)),\delta)\subset R(F)$, we obtain $Sh(F)\subset(\Sigma_m^+\times X)^{deg}$ by Lemma \ref{lemma3.6}. From the proof process of Lemma \ref{lemma3.5}, it can be seen that $Sh(G)\subset X^{deg}$. $\hfill\square$

   \quad The following example will state the rationality of the condition in Theorem \ref{thm1.2}.
   \begin{example}
   Let $X=\mathbb{S}^1$ is one-dimensional torus, $f_0,f_1$ are all irrational rotation with irrational independence. Assume $f_0(x)=x+\alpha_0~(mod~1)$ and  $f_1(x)=x+\alpha_1~(mod~1)$ with $\alpha_0+\alpha_1\in\mathbb{R}\textbackslash\mathbb{Q}$. The periodic point $\omega=(0,1)^{\infty}$ is a minimal point in the dynamical system  $(\Sigma_2^+,\sigma)$. For minimal sub-system $(\overline{Orb(\omega,\sigma)},\sigma)$, define skew-product transformation $F: \overline{Orb(\omega,\sigma)} \times \mathbb{S}^1 \to \overline{Orb(\omega,\sigma)} \times \mathbb{S}^1$ corresponding to $\{f_0,f_1\}$. There exists $\delta>0$ such that  $B(\gamma(Sh(F)),\delta)\subset R(F)$ according to the following proof.

   Clearly $\overline{Orb(\omega,\sigma)}=\{\omega,\sigma(\omega)\}$, then we prove that $R(F)=\overline{Orb(\omega,\sigma)}\times \mathbb{S}^1$. Take $x\in\mathbb{S}^1$. For $(\omega,x)\in\overline{Orb(\omega,\sigma)}\times \mathbb{S}^1$, define $g_0(x)=f_1\circ f_0(x)=x+\alpha_0+\alpha_1~(mod~1)$. Since $\alpha_0+\alpha_1\in\mathbb{R}\textbackslash\mathbb{Q}$, we obtain $g_0$ is minimal, that is $R(g_0)=\mathbb{S}^1$. For $x\in R(g_0)$, there exists an increasing sequence $\{n_k\}_{k\ge0}$ of positive integers such that $\lim\limits_{k\to\infty}g_0^{n_k}(x)=x$, then we have
   \[
   \lim\limits_{k\to\infty}F^{2n_k}(\omega,x)=\lim\limits_{k\to\infty}(\sigma^{2n_k}(\omega),f_{\omega}^{2n_k}(x))=\lim\limits_{k\to\infty}(\omega,g_0^{n_k}(x))=(\omega,x),
   \]
   that is $(\omega,x)\in R(F)$. For $(\sigma(\omega),x)$, define $g_1(x)=f_0\circ f_1(x)=x+\alpha_1+\alpha_1~(mod~1)$. Similarly we have $R(g_1)=\mathbb{S}^1$. For $x\in R(g_1)$, there exists an increasing sequence $\{n_i\}_{i\ge0}$ of positive integers such that $\lim\limits_{i\to\infty}g_1^{n_i}(x)=x$, then we have
   \[
   \lim\limits_{i\to\infty}F^{2n_i}(\sigma(\omega),x)=\lim\limits_{i\to\infty}(\sigma^{2n_i+1}(\omega),f_{\sigma(\omega)}^{2n_i}(x))=\lim\limits_{i\to\infty}(\sigma(\omega),g_1^{n_i}(x))=(\sigma(\omega),x),
   \]
   that is $(\sigma(\omega),x)\in R(F)$. In all we have $R(F)=\overline{Orb(\omega,\sigma)}\times \mathbb{S}^1$, so there exists $\delta>0$ such that $B(\gamma(Sh(F)),\delta)\subset R(F)$.
   \end{example}

   \section{Declarations}
   We declare that we have no financial and personal relationships with other people or organizations that can inappropriately influence our work, there is no professional or other personal interest of any nature or kind in any product, service or company that could be construed as influencing the position presented in, or the review of, the manuscript entitled, "Shadowable points of free semigroup actions".

   \nocite{MR2831912}
   \nocite{MR2129113}
   \nocite{MR3842255}
   \nocite{MR1727170}
   \nocite{MR0648108}
   \bibliographystyle{abbrv}
   \bibliography{reference0}

\end{document}